\documentclass[12pt,reqno]{amsart}
\usepackage{amssymb}
\usepackage{amsmath}
\usepackage{mathrsfs}
\usepackage{amsfonts}
\usepackage{mathabx}
\usepackage{color}
\usepackage[hidelinks]{hyperref}
\usepackage{extarrows}

\makeatletter
\def\tank#1{\protected@xdef\@thanks{\@thanks
 \protect\footnotetext[0]{#1}}}
\def\bigfoot{

 \@footnotetext}
\makeatother

\newcommand{\ea}{\end{array}}

\allowdisplaybreaks

\newtheorem{theorem}{Theorem}[section]
\newtheorem{lemma}{Lemma}[section]

\newtheorem{corollary}[theorem]{Corollary}
\newtheorem{definition}[theorem]{Definition}
\newtheorem{example}[theorem]{Example}

\newtheorem{rem}{Remark}[section]

\def\beq{\begin{equation}}
\def\nneq{\end{equation}}

\def\bthm{\begin{theorem}}
\def\nthm{\end{theorem}}

\def\blem{\begin{lemma}}
\def\nlem{\end{lemma}}

\def\bprf{\begin{proof}}
\def\nprf{\end{proof}}

\def\bprop{\begin{prop}}
\def\nprop{\end{prop}}

\def\brmk{\begin{rem}}
\def\nrmk{\end{rem}}

\def\bexa{\begin{example}}
\def\nexa{\end{example}}

\def\bcor{\begin{corollary}}
\def\ncor{\end{corollary}}

\newcommand{\ee}{\mathbb{E}}

\newcommand{\rr}{\mathbb{R}}
\newcommand{\pp}{\mathbb{P}}
\newcommand{\qq}{\mathbb{Q}}

\newcommand{\Lip}{{\mathrm{{\rm Lip}}}}

\def\e{\varepsilon}

\newcommand{\gcbl}[1]{\mathrm{GCB}\!\left(#1\right)}

\title[Talagrand's transportation inequality  for SPDEs with monotone drifts]{Talagrand's transportation inequality  for SPDEs with locally monotone drifts}

\author{Ruinan Li}
\address{Ruinan Li \\School of Statistics and Information, Shanghai University of International Business and Economics, Shanghai 201620,  PR China.}
\email{ruinanli@amss.ac.cn}

\author{Xinyu Wang}
\address{Xinyu Wang  \\Wenlan School of Business, Zhongnan University of Economics and Law, Wuhan 430073, P. R. China}
\email[Corresponding author]{wang\_xin\_yu@zuel.edu.cn}

\date{}
\begin{document}
\maketitle

\noindent{\bf Abstract}
The purpose of this paper is twofold. Firstly, we prove transportation inequalities ${\bf T_2}(C)$  on the space of continuous paths with respect to the uniform metric for the law of the solution to a class of non-linear monotone stochastic partial differential equations (SPDEs) driven by the Wiener noise. Furthermore, we also establish the  ${\bf T_1}(C)$ property for such SPDEs  but with merely locally monotone  coefficients,  including the stochastic Burgers type equation and stochastic $2$-D Navier-Stokes equation.

\vskip0.3cm
\noindent {\bf Keywords}{ Stochastic  partial differential equation, Transportation inequality, Locally monotone, Girsanov transformation.}

\vskip0.3cm

\noindent {\bf Mathematics Subject Classification (2010)}{ 60E15; 60H15.}
\maketitle


\section{Introduction}

Let $(E,d)$ be a metric space with Borel $\sigma$-algebra $\mathcal{B}(E)$, and $\mathcal M(E)$ be the space of all Borel probability measures on $E$. We say that $\mu\in \mathcal{M}(E)$ has normal concentration on $(E,d)$ if there are constants $C$, $r>0$ such that for every $\e>0$ and every Borel subset $A\in \mathcal{B}(E)$ with $\mu(A)\geq\frac{1}{2}$,
\beq \label{constration}
\mu(A_{\e})\geq 1-Ce^{-r\e^2},
\nneq
where $A_{\e}$ is the $\e$-neighborhood of $A$, i.e., $$A_{\e}=\{y:d(x,y)<\e,\text{ for some }  x\in A \}.$$
The concentration of measure phenomenon has wide applications, e.g. to stochastic
finance (see \cite{La}), statistics (see \cite{M}) and the analysis of randomized algorithms (see \cite{DP}).
In the past decades, many people established normal concentration properties for various kinds of  measures.
It is well known that one way to prove \eqref{constration} is to show a stronger statement, called {\it Talagrand inequality}, or
{\it transportation cost information inequality} (TCI inequality). These are inequalities that compare the Wasserstein distance with the relative entropy.
Let us recall some relevant concepts.
Fix a real number $p\ge1$, the Wasserstein distance of order $p$ between two probability measures $\mu,\nu\in \mathcal M(E)$ is defined as
\[
 W_{p}(\mu, \nu):=\inf_{\pi}\left[\int_{E}\int_E d(x,y)^p \, \pi(dx,dy) \right]^{\frac1p},
\]
where the infimum is taken over all  the probability measures  $\pi$  on $E\times E$ with marginal distributions $\mu$ and $\nu$
(saying couplings of $\mu, \nu$).  The relative entropy (or the Kullback information) $H(\nu|\mu)$ of $\nu$ with respect to (w.r.t. for short) $\mu$ is defined as
\begin{equation}\label{rate function1}
 H(\nu|\mu):=\left\{
       \begin{array}{ll}
         \int \log \frac{d\nu}{d\mu}d\nu,   & \text{if } \nu\ll \mu ;\\
        +\infty, & \hbox{\text{otherwise}.}
       \end{array}
     \right.
\end{equation}
\begin{definition}
The probability measure $\mu$ is said to satisfy the $T_p$ transportation cost information
inequality (also called the Talagrand inequality) on $(E, d)$ if there exists a constant $C>0$ such that for any probability measure $\nu$ on $E$,
\begin{equation}\label{TCI1}
W_{p}(\mu,\nu)\le \sqrt{2C  H(\nu|\mu)}.
\end{equation}
\end{definition}
As usual, we write $\mu\in {\bf T_{p}}(C)$ for this relation.
The properties ${\bf T_1}(C)$ and ${\bf T_2}(C)$ are particularly interesting.
That ${\bf T_1}(C)$ is related to the phenomenon of the measure concentration was emphasized by Marton
\cite{Marton96, Marton97}, Talagrand \cite{Tal0, Tal1, Tal2}, Bobkov and G\"{o}tze \cite{BG} and amply explored by Ledoux \cite{Ledoux01,Ledoux03}.
We refer the readers to the monograph \cite{BH2000,BLM, Ledoux01, Tal1} for nice expositions of the concentration of measure phenomenon.
The ${\bf T_2}(C)$, referred to as the quadratic transportation cost inequality, was first established by Talagrand \cite{Tal1} for the Gaussian measure with the sharp constant $C = 2$. It has been brought into relation with the log-Sobolev inequality, Poincar\'{e} inequality, inf-convolution, Hamilton-Jacobi equations by Otto and Villani \cite{OV, Villani} and Bobkov, Gentil and Ledoux \cite{BGL2001}. According to H\"{o}lder's inequality, if we increase $p$ in \eqref{TCI1}, this inequality becomes stronger. Hence $\mu\in {\bf T_{2}}(C)$ implies $\mu\in {\bf T_{1}}(C)$.

For several years, the problem of  transportation cost information to stochastic (partial) differential equations has been widely studied  and is still a very active research area from both a theoretical and an applied point of view, see \cite{DP, Goz, La} and references therein.  The TCIs for stochastic differential equations were obtained by H. Djellout, A. Guillin and L. Wu in \cite{DGW}.
L. Wu and Z. Zhang \cite{WZ2006} studied the ${\bf T_2}(C)$  for SPDEs w.r.t. the $L^2$-norm by Galerkin's approximation. By Girsanov's transformation, B. Boufoussi and S. Hajji \cite{BH} obtained the ${\bf T_2}(C)$ w.r.t. the $L^2$-metric for the stochastic heat equations driven by the space-time white noise and driven by the fractional-white noise.  Khoshnevisan and  Sarantsev  \cite{KS} established  the ${\bf T_2}(C)$ for more general SPDEs driven by the space-time white noise under the uniform and $L^2$-norm. S. Shang and T. Zhang  \cite{SZ}  established the  ${\bf T_2}(C)$ w.r.t. the uniform metric for the stochastic heat equation driven by the multiplicative space-time white noise,  which was extended to the time-white and space-colored noise  case by S. Shang  and R. Wang  \cite{SW}. Y. Dai and R. Li \cite{DL} proved the ${\bf T_2}(C)$ for stochastic heat equation with rough dependence in space w.r.t. the weighted   $L^2$-norm. F.-Y. Wang and T. Zhang  \cite{WZ} studied the   ${\bf T_2}(C)$ for SPDEs with random initial values.   Y. Li and X. Wang \cite{LW1}  established the   ${\bf T_2}(C)$ w.r.t. the weighted   $L^2$-norm  for stochastic wave equation on   $\mathbb R^3$.  Y. Ma and  R. Wang \cite{MW} studied the ${\bf T_1}(C)$ for stochastic reaction-diffusion equations with L\'evy noises.

The aim of this paper is to prove that the quadratic transportation cost inequality holds to a class of non-linear monotone SPDEs under the uniform distance  and the ${\bf T_1}(C)$ property holds under the $L^2$-metric to a class of non-linear locally monotone SPDEs. After we finished this article, we found a closely related article \cite{KA} which proved a ${\bf T_2}(C)$  for the solution of evolutionary $p$-Laplace equation w.r.t. the $L^2$-metric. However, the paper does not contain a value of the constant $C$.

The rest of the paper is organized as follows. In Sect. 2, we establish the quadratic transportation cost inequality ${\bf T_2}(C)$  for the law of the solution to the SPDEs  with monotone drifts. In Sect. 3, we prove the ${\bf T_1}(C)$ for  the SPDEs  with only locally monotone drifts, with  the stochastic Burgers type equation and stochastic $2$-D Navier-Stokes equation provided as examples.

\section{Quadratic transportation cost inequalities for  SPDEs  with monotone drifts}

\subsection{SPDEs with monotone drifts}

Let $(H, \langle \cdot,\cdot\rangle_H, \|\cdot\|_H)$ be a separable Hilbert space and  $(V,\|\cdot\|_V)$  a Banach space such that $V\subset H$ continuously and densely. Let $V^*$ be the dual space of $V$, it is well known
$$
V\subset H\subset V^*
$$
continuously and densely. If $_{V^*}\langle\cdot, \cdot\rangle_{V}$ denotes the dualization between $V^*$ and $V$, it follows that
$$
_{V^*}\langle z, v\rangle_{V}=
\langle z, v\rangle_{H}\ \ \ \ \text{for all } z\in H, v\in V.
$$
$(V, H, V^*)$ is called a {\it Gelfand triple}.
In this paper, we always assume that $V$ is compactly embedded in $H$.
Thus, there exists a constant $\eta > 0$ such that
\begin{equation}\label{norm-rel}
\|x\|_V\geq \eta \|x\|_H \ \ \text{for all }x\in V.
\end{equation}

 Let  $\{W_t\}_{t\ge0}$ be a cylindrical Wiener process  on a  separable Hilbert space $(U,\langle \cdot, \cdot\rangle_U)$  w.r.t.  a  complete probability space $(\Omega, \mathcal F, \{\mathcal F_t\}_{t\geq 0}, \pp)$  and $(\mathcal L_2(U;H), \|\cdot\|_{\mathcal L_2(U;H)})$ be the space of all Hilbert-Schmidt operators from $U$ to $H$. Now we consider the following stochastic evolution equation
\begin{equation}\label{eq SPDE}
d X_t=A(t, X_t)d t+B(t, X_t)d W_t, \ \ \  X_0=x\in H,
\end{equation}
where
$$
A: [0,T]\times V\rightarrow V^*; \ \ \  B:[0,T]\times V\rightarrow \mathcal L_2(U; H).
$$

Assume that  there exist constants $\alpha>1$, $\theta>0$, $K_2, K_3\in \mathbb R$, $K_4>0$
and a positive adapted process $f\in L^1([0,T]\times \Omega; dt\times \pp)$ such that the following conditions hold
for all  $v, v_1,v_2\in V$ and $t\in[0,T]$:
 \begin{itemize}
  \item[{\bf (H1)}]  Hemicontinuity of $A$:  the map
  $$\lambda\mapsto   {_{V^*}\langle }A(t, v_1+\lambda v_2), v\rangle_{V}$$
  is continuous on $\mathbb R$.
  \item[{\bf (H2)}]  Monotonicity of $(A,B)$:
  $$
  2 _{V^*}\langle A(t, v_1)-A(t, v_2), v_1-v_2\rangle_{V}+\|B(t, v_1)-B(t, v_2)\|_{\mathcal L_2(U;H)}^2\le K_2 \|v_1-v_2\|_H^2.
  $$
  \item[{\bf (H3)}] Coercivity of $(A,B)$:
  $$
  2 _{V^*}\langle A(t, v), v\rangle_{V}+\|B(t, v)\|_{\mathcal  L_2(U;H)}^2\le f_t -\theta \|v\|_V^{\alpha}+  K_3\|v\|_H^2.
  $$
  \item[{\bf (H4)}] Boundedness of $A$:
  $$
  \|A(t,  v)\|_{V*}\le  f_t^{(\alpha-1)/\alpha}+K_4 \|v\|_V^{\alpha-1}.
  $$
  \end{itemize}

According to \cite{KR}, we have the following result about the existence and uniqueness of  solutions to Eq.\eqref{eq SPDE}.
\begin{lemma}
Assume that $({\bf H1})$-$({\bf H4})$ hold. For any $X_0\in L^2(\Omega, \mathcal F_0, \mathbb P; H)$, \eqref{eq SPDE} admits a unique solution $X=\{X_t\}_{t\in[0,T]}$ in $C([0,T];H)\cap L^{\alpha}([0,T];V)$, which is an adapted continuous process on $H$ such that
 \begin{equation}\label{eq solution 1}
 \langle X_t, v\rangle_H=\langle X_0, v\rangle_H+ \int_0^t {_{V^*}}\langle A(s, X_s), v\rangle_Vds+\int_0^t \langle B(s,X_s)dW_s, v\rangle_H
 \end{equation}
 holds for all $v\in V$ and $(t,\omega)\in [0, T]\times \Omega$. Furthermore,
 \begin{equation}\label{eq solution 2}
 \mathbb E\left[\sup_{0\le t \le T} \|X_t\|_{H}^2+\int_0^T\|X_t\|_{V}^{\alpha}dt \right]<\infty.
 \end{equation}
\end{lemma}
The formulation of  \eqref{eq SPDE}  includes a lot of interesting examples
of SPDEs, such as stochastic $p$-Laplace equation (\cite{Liu}, \cite{PR}),
stochastic generalized porous media equation (\cite{DRR}, \cite{RRW}) and stochastic fast-diffusion equation (\cite{LW}, \cite{RRW}).
The reader can refer to the related references listed above for more details.

\subsection{${\bf T_2}(C)$   for  SPDEs  with monotone drifts }

Let $\{X_t\}_{t\in[0,T]}$ be the unique solution of equation \eqref{eq SPDE}.
For any  $\mu\in \mathcal M(H)$, let $P^{\mu}$ be the distribution of the solution  $\{X_t\}_{t\in[0,T]}$ on $C([0,T];H)$ such that the law of $X_0$ is $\mu$.  Particularly, if $\mu=\delta_{x}$ for some $x\in H$, we write $P^{x}:=P^{\delta_{x}}$ for short.

In order to investigate the  quadratic transportation inequality for SPDEs  with monotone drifts, we further assume that the following condition hold.
 \begin{itemize}
 \item[{\bf (H5)}] There exists a positive constant $K_T$ such that
 $$
 C_B:=\sup_{t\in[0,T]}\sup_{u\in H} \|B(t, u)\|^2_{\mathcal  L_2(U;H)}\le  K_{T}.
 $$

 \end{itemize}

Here is the main result of this section.
\bthm\label{thm transport 1}
Under  $({\bf H1})$-$({\bf H5})$, for any deterministic initial value $x\in H$,  the law $P^{x}$ satisfies the $T_2$-transportation inequality on the space $C([0,T];H)$  w.r.t. the uniform  metric with the constant $C(T, K_2, C_B)$ defined by
\begin{equation}\label{Costa1}
C(T, K_2, C_B):=\inf_{0<\varepsilon_1+\varepsilon_2<1}\frac{C_B}{\varepsilon_1(1-\varepsilon_1-\varepsilon_2)}
e^{\frac{\left(\varepsilon_2+C_1^2\right)K_2T}{(1-\varepsilon_1-\varepsilon_2)\varepsilon_2}}.
\end{equation}
\nthm

We will apply the Girsanov theorem to prove Theorem \ref{thm transport 1}.
First we state a lemma  describing all probability measures $Q$ that are absolutely continuous w.r.t. $P^{x}$.
It is analogous to \cite[Theorem 5.6]{DGW} in the setting of finite-dimensional Brownian motion,
\cite[Lemma 3.1]{KS} in the setting of the space-time white noise on $\rr^2$ and \cite[Lemma 3.1]{SW} in the setting of Gaussian noise white in
time and colored in space, hence we omit its proof.

Let $Q$ be a probability measure on $C([0,T];H)$ such that $Q\ll P^{x}$.
Define a new probability measure $\mathbb Q$ on the filtered probability space $(\Omega,\mathcal{F}, \{\mathcal{F}\}_{0\leq t\leq T}, \mathbb{P})$ by
\begin{equation}\label{eq Q}
  d\mathbb Q:=\frac{dQ}{d P^{x}}(u)d\mathbb P,
\end{equation}
and let $\ee^{\mathbb Q}$ denotes the expectation w.r.t. this new measure.
Denote the Radon-Nikodym derivative restricted on $\mathcal F_t$ by
\begin{equation*}
M_t := \left. \frac{d\mathbb Q}{d\mathbb P}\right|_{\mathcal F_t},\quad t\in [0,T].
\end{equation*}
Then $\{M_t \}_{0\le t\le T}$ is a nonnegative $\mathbb P$-martingale.
\blem\label{lem Girsanov}
There exists an $U$-valued predictable process  $h=\{h(t), t\in[0,T] \}$ such that
\begin{align}
	\int_0^T \|h(s)\|_{U}^2ds<\infty, \quad \mathbb Q\text{-a.s.},
\end{align}
and the process
\begin{align}\label{190725.224510}
	\widetilde W_t := W_t -\int_0^t h(s)ds, \quad t\in[0,T],
\end{align}
is a cylindrical Wiener process on $U$ under  $\mathbb Q$. Moreover,
\beq\label{Mar}
M_t =\exp\left(\int_0^t h(s)dW_s-\frac{1}{2} \int_0^t\|h(s)\|_{U}^2ds\right)\ \quad \mathbb Q\text{-a.s.},
\nneq
and
\beq\label{eq entropy}
H(Q|P^{x})=\frac{1}{2} \ee^{\mathbb Q} \left[\int_0^T\|h(s)\|_{U}^2ds\right].
\nneq

\nlem
Now we are in the position to prove the main result of this section.
\begin{proof}[The proof of Theorem \ref{thm transport 1}]

For any $Q\ll P^{x}$ such that $ H(Q|P^{x})<\infty$, let $\mathbb Q$ be defined as (\ref{eq Q}) and $h$ be the corresponding process appeared in Lemma \ref{lem Girsanov}. Then it is easy to see that the solution of equation \eqref{eq SPDE} satisfies the following SPDE under the measure $\mathbb Q$:
\begin{equation}\label{eq SPDE1}
d X_t=A(t, X_t)d t+B(t, X_t)d\widetilde W_t+B(t, X_t) h(t)dt, \ \ \  X_0=x\in H.
\end{equation}
Consider the solution of the following equation:
\begin{equation}\label{eq SPDE2}
d Y_t=A(t, Y_t)d t+B(t, Y_t)d\widetilde W_t, \ \ \  Y_0=x\in H.
\end{equation}
 By Lemma \ref{lem Girsanov}, it follows that under the probability measure $\mathbb Q$, $(Y, X)$ forms a coupling of $(P^x, Q)$. Therefore, by the definition of the Wasserstein distance, we have
\beq\label{eq XY}
W_{2}^2(Q, P^{x}) \le \ee^{\mathbb Q}\left[\sup_{t\in [0,T]} \|X_t-Y_t\|_H^2 \right].
\nneq
In light of \eqref{eq entropy} and \eqref{eq XY},  to prove the $T_2$-transportation inequality, it is sufficient to show that
\beq\label{eq tar}
\ee^{\mathbb Q}\left[\sup_{t\in [0,T]} \|X_t-Y_t\|_H^2 \right]\le C \ee^{\qq}\left[\int_0^T\|h(s)\|_{U}^2ds\right],
\nneq
for some constant $C$ independent of $h$.

From \eqref{eq SPDE1} and \eqref{eq SPDE2}, by It\^o's formula and   ({\bf H2}) , we have
\begin{align}
\|X_t-Y_t\|_H^2=& 2 \int_0^t {_{V^*}}\left\langle A(s, X_s)-A(s, Y_s), X_s-Y_s\right \rangle_V ds\notag\\
&+ 2 \int_0^t \left\langle   X_s-Y_s,   \left[B(s, X_s)-B(s, Y_s)\right] d\widetilde W_s \right \rangle_H \notag\\
&+ 2 \int_0^t \left\langle   X_s-Y_s,   B(s, X_s) h(s) ds\right \rangle_H \notag+\int_0^t \left\|B(s,X_s)-B(s,Y_s) \right\|_{\mathcal  L_2(U;H)}^2ds\notag\\
\le & K_2\int_0^t \left\|X_s-Y_s \right\|_H^2ds +
2 \int_0^t \left\langle   X_s-Y_s,   \left[B(s, X_s)-B(s, Y_s)\right] d\widetilde W_s \right \rangle_H \notag\\
&+ 2 \int_0^t \left\langle   X_s-Y_s,  B(s, X_s) h(s) ds\right \rangle_H.
\end{align}
By the Burkhold-Davis-Gundy inequality,    ({\bf H2}),  ({\bf H5}) and the  inequality
 \begin{align}\label{eq elem}
 2ab\le \varepsilon a^2+\frac{1}{\varepsilon} b^2,  \ \ \ \ \forall a, b,\varepsilon >0,
 \end{align}
we have that for some positive constant $C_1$ and any $\varepsilon_1, \varepsilon_2>0$,
\begin{align}\label{eq XY2}
&\mathbb E^{\mathbb Q}\left[\sup_{s\in[0,t]}\|X_s-Y_s\|_H^2\right] \notag\\
\le & K_2\mathbb E^{\mathbb Q}\left[\int_0^t \left\|X_s-Y_s \right\|_H^2ds\right] +
2C_1 \mathbb E^{\mathbb Q}\left[\left(\int_0^t   \left\| B(s, X_s)-B(s, Y_s)\right\|_{\mathcal  L_2(U;H)}^2\left\|X_s-Y_s\right\|_H^2ds\right)^{\frac12} \right]\notag\\
&+ \varepsilon_1\mathbb E^{\mathbb Q}\left[\sup_{s\in[0,t]}\|X_s-Y_s\|_H^2 \right]+
\frac{\sup_{t\in [0,T], u\in H}\|B(t,u)\|_{\mathcal  L_2(U;H)}^2}{\varepsilon_1}\mathbb E^{\mathbb Q}\left[\int_0^T\|h(s)\|_{U}^2ds\right]\notag\\
\le & K_2\mathbb E^{\mathbb Q} \left[\int_0^t \left\|X_s-Y_s \right\|_H^2ds\right] +
2C_1 \sqrt{K_2} \mathbb E^{\mathbb Q}\left[\left( \int_0^t  \left\|X_s-Y_s\right\|_H^4ds\right)^{\frac12}\right] \notag\\
&+ \varepsilon_1\mathbb E^{\mathbb Q}\left[\sup_{s\in[0,t]}\|X_s-Y_s\|_H^2\right]  + \frac{C_B}{\varepsilon_1}\mathbb E^{\mathbb Q}\left[\int_0^T\|h(s)\|_{U}^2ds\right]\notag\\
\le & K_2\mathbb E^{\mathbb Q} \left[\int_0^t \left\|X_s-Y_s \right\|_H^2ds\right]
+2C_1 \sqrt{K_2}\mathbb E^{\mathbb Q}\left( \sup_{s\in[0,t]}\|X_s-Y_s\|_H \cdot \left[\int_0^t \left\|X_s-Y_s\right\|_H^2ds\right]^{\frac12}\right) \notag\\
&+ \varepsilon_1\mathbb E^{\mathbb Q}\left[\sup_{s\in[0,t]}\|X_s-Y_s\|_H^2\right] + \frac{C_B}{\varepsilon_1} \mathbb E^{\mathbb Q}\left[\int_0^T\|h(s)\|_{U}^2ds\right]\notag\\
\le &
\left(K_2+\frac{C_1^2 K_2}{\varepsilon_2}\right) \mathbb E^{\mathbb Q}  \left[\int_0^t \left\|X_s-Y_s\right\|_H^2ds\right]
+( \varepsilon_1+\varepsilon_2)\mathbb E^{\mathbb Q} \left[\sup_{s\in[0,t]}\|X_s-Y_s\|_H^2\right] + \frac{C_B}{\varepsilon_1} \mathbb E^{\mathbb Q}\left[\int_0^T\|h(s)\|_{U}^2ds\right].\notag\\
\end{align}
Choosing $\varepsilon_1, \varepsilon_2$ small enough such that  $\varepsilon_1+\varepsilon_2<1$, we have
$$
\mathbb E^{\mathbb Q}  \left[\sup_{s\in[0,t]}\|X_s-Y_s\|_H^2\right]\leq \frac{K_2+\frac{C_1^2 K_2}{\varepsilon_2}}{1-\varepsilon_1-\varepsilon_2}\mathbb E^{\mathbb Q} \left[\int_0^t \left\|X_s-Y_s\right\|_H^2ds\right]+ \frac{C_B}{\varepsilon_1(1-\varepsilon_1-\varepsilon_2)} \mathbb E^{\mathbb Q} \left[\int_0^t \|h(s)\|_{U}^2 ds\right].
$$
Using the Gronwall inequality, we obtain
\begin{align}
\mathbb E^{\mathbb Q} \left[\sup_{t\in[0,T]}\|X_t-Y_t\|_H^2\right] \le C(T, K_2, C_B) \mathbb E^{\mathbb Q}\left[\int_0^T\|h(s)\|_{U}^2ds\right],
\end{align}
where
\begin{equation}\label{Costa}
C(T, K_2, C_B)=\inf_{0<\varepsilon_1+\varepsilon_2<1}\frac{C_B}{\varepsilon_1(1-\varepsilon_1-\varepsilon_2)}
e^{\frac{\left(\varepsilon_2+C_1^2 \right)K_2T}{(1-\varepsilon_1-\varepsilon_2)\varepsilon_2}}<+\infty.
\end{equation}
\end{proof}


Applying  Theorem \ref{thm transport 1} and using the same approach as that in the proof of \cite[Theorem 3.1]{WZ}, we can get the following quadratic transportation cost inequality for the monotone  SPDE \eqref{eq SPDE} with random initial values, whose proof is omitted here.
\bcor
Let $\mu \in \mathcal M(H)$. Under  $({\bf H1})$-$({\bf H5})$,
\beq
W_{2}^2(Q, P^{\mu})\leq CH(Q|P^{\mu}),\ \ \ \ \forall Q \in \mathcal M(C([0,T];H))
\nneq
holds for some constant $C > 0$ if and only if
\beq
W_{2}^2(\nu, \mu)\leq c H(\nu|\mu),\ \ \ \ \forall \nu \in \mathcal M(H).
\nneq
holds for some constant $c > 0$.
\ncor

\section{ ${\bf T_1}(C)$  for  SPDEs  with local  monotone drifts and applications}

In this section, we will investigete the ${\bf T_1}$ property for Eq.\eqref{eq SPDE} under merely locally monotone coefficients.
This general framework is conceptually not more involved than the classical one, but includes many more fundamental examples not included previously.

\subsection{SPDEs with local monotone drifts}
Let us first state the new conditions on the coefficients of Eq.\eqref{eq SPDE}.
Instead of $({\bf H2})$ and $({\bf H4})$ in the previous section,  we assume that there exist constants $\alpha>1$, $\beta\geq 0$, $\tilde{K}_2$, $\tilde{K}_4\in \mathbb R$ and a positive adapted process $\tilde{f}\in L^1([0,T]\times \Omega; dt\times \pp)$ such that the following conditions hold for all  $v, v_1,v_2\in V$ and $t\in[0,T]$:
\begin{itemize}
  \item[({\bf H2'})]  Local monotonicity of $(A,B)$:
  \begin{align}
  &2 _{V^*}\langle A(t, v_1)-A(t, v_2), v_1-v_2\rangle_{V}+\|B(t, v_1)-B(t, v_2)\|_{\mathcal L_2(U;H)}^2\notag\\
  \le&  \left(\tilde{K}_2+\rho(v_2)\right)\|v_1-v_2\|_H^2,
  \end{align}
  where  $\rho: V\rightarrow [0,\infty)$ is a measurable function and locally bounded in $V$.
  \item[({\bf H4'})] Growth of $A$:
  $$
  \|A(t,  v)\|_{V*}^{\frac{\alpha}{\alpha-1}}\le  \left(\tilde{f}_t+\tilde{K}_4\|v\|_{V}^{\alpha}\right)\left(1+\|v\|_H^{\beta}\right).
  $$
  \end{itemize}

The existence and uniqueness of solutions to SPDE \eqref{eq SPDE} with locally monotone coefficients is due to the following result in \cite{LR2010}.

\blem\label{local solu}(\cite[Theorem 1.1]{LR2010})
Assume that ({\bf H1}), ({\bf H2'}), ({\bf H3}), ({\bf H4'}), ({\bf H5}) hold for $\tilde{f}\in L^{p/2}([0,T]\times \Omega; dt\times \pp)$ with some $p\geq \beta+2$ and there exists a constant $C$ such that
\beq\label{local p}
\rho(v)\leq C\left(1+\|v\|_{V}^{\alpha}\right)\left(1+\|v\|_H^{\beta}\right),\ \ v\in V.
\nneq
Then for any $X_0\in L^p(\Omega, \mathcal F_0, \mathbb P; H)$, Eq.\eqref{eq SPDE} has a unique solution $\{X_t\}_{t\in[0,T]}$ and satisfies

\begin{equation}\label{eq solution 3}
 \mathbb E\left[\sup_{0\le t \le T} \|X_t\|_{H}^p+\int_0^T\|X_t\|_{V}^{\alpha}dt \right]<\infty.
 \end{equation}
\nlem
This formulation  also includes a lot of interesting examples, such as stochastic reaction-diffusion
equations, stochastic Burgers type equation, stochastic $2$-D Navier-Stokes equation, stochastic $p$-Laplace
equation and stochastic porous media equation with non-monotone perturbations, see \cite{LR2010}.

\subsection{${\bf T_1(C)}$   for  SPDEs  with local monotone drifts}
In order to investigate the $T_1$-transportation inequality for SPDEs  with local monotone drifts, we focus on the case of $\alpha=2$ in ({\bf H3}), ({\bf H4'}) and \eqref{local p}. We also assume that
\beq\label{cons-con}
\theta \eta-K_3>0,
\nneq
where $\eta$ is the constant in Eq.\eqref{norm-rel}, and $\theta$ and $K_3$ are the constants in {\bf (H3)}.

For any $\lambda_0>0$, let
$$
\mathcal M_{\lambda_0}:=\left\{\mu \in \mathcal M(H): 1\leq \int_H e^{\lambda_0 \|x\|_H^2}\mu(dx)<\infty\right\}.
$$

The main result of this section is the following theorem, which provides the ${\bf T_1}$  property of  $P^{\mu}$
w.r.t. the $L^{2}$-metric  in  $L^{2}([0,T];V)$.
\bthm\label{T1}
Assume that all the conditions in Lemma \ref{local solu} are satisfied with $\alpha=2$ and a deterministic function $\tilde{f}$. Furthermore, we assume that \eqref{cons-con} holds. Then for any  $0<c<1-\frac{K_3}{\theta\eta}$, when $\mu\in \mathcal M_{\lambda_0}$ with $\lambda_0\in\left(0,\frac{(1-c)\theta \eta-K_3}{2C_B}\right)$, the law $P^{\mu}$ satisfies the $T_1$-transportation inequality on the space $L^{2}([0,T];V)$  w.r.t. the $L^{2}$-metric   with the constant $C$ given by
\beq\label{T1const}
C=\frac{\exp\left(2\lambda_0\int_0^T \tilde{f}(s)ds+1\right)}{c\lambda_0 \theta \sqrt{\pi}} \left(\int_H e^{\lambda_0\|x\|_H^2}\mu(dx)\right)^2.
\nneq
\nthm
The proof of this result is based on a general equivalence  due to \cite{BG} and is carried out in the next subsection.

\subsection{A general equivalence}
Let $(E,d)$  be a metric space. We call a function $F: E\rightarrow \mathbb R$ Lipschitz, and denote by $F\in\Lip(E;\mathbb R)$,
if
$$
\|F\|_{\mathrm Lip}: =\sup_{x\neq y}\frac{|F(x)-F(y)|}{d(x,y)}<+\infty.
$$

\begin{theorem}(Bobkov and G\"otze \cite{BG})\label{BG-T}
A probability measure $\mu$ satisfies the ${\bf T_1}(C)$ on  $(E,d)$ with constant $C>0$, if and only if for any $F\in\Lip(E;\mathbb R)$, $F$ is $\mu$-integrable and
\begin{equation}\label{eq T1 1}
\int e^{\lambda (F-\mu(F))}d\mu\le \exp\left(\frac{C}{2}\lambda^2\|F\|_{{\mathrm Lip}}^2 \right),  \ \ \ \forall  \lambda\in \mathbb R.
\end{equation}

\end{theorem}
According to Theorem \ref{BG-T}, to establish the ${\bf T_1}(C)$ property for a probability measure $\mu$ it is enough to verify Eq.\eqref{eq T1 1}, namely $\mu$ satisfies a Gaussian concentration bound defined  as follows.

\begin{definition}
Let $\mu$ be a probability measure on  $(E,d)$.
We say that $\mu$ satisfies a Gaussian concentration bound with constant $D>0$ on $(E,d)$ if there exists $x_0\in E$ such that $\int d(x_0,x) d\mu(x)<+\infty$  and for all $f\in\Lip(E;\mathbb R)$, one has
\[
\int e^{f-\mu(f)} d \mu \leq e^{D \|f\|_{\mathrm Lip}^2}.
\]
For brevity we shall say that $\mu$ satisfies $\gcbl{D}$ on $(E,d)$.
\end{definition}

The following result  providing an  equivalence between the Gaussian concentration and distance Gaussian moment bounds is  established in \cite{CCR}.
\begin{theorem}\label{thmgaussianexpmoment}(\cite[Theorem 2.1]{CCR})
Let $\mu$ be a probability measure on $(E,d)$.
Then $\mu$ satisifies a Gaussian concentration bound if and only if it has a Gaussian moment.
More precisely, we have the following:
\begin{enumerate}
\item
If $\mu$ satisfies $\gcbl{D}$, then there exists $x_0\in E$ such that
\begin{equation}\label{leiden1}
\int e^{\frac{d(x_0,x)^2}{16D}} d\mu(x) \leq 3 e^{\frac{\mu(d)^2}{D}}
\end{equation}
where $\mu(d):=\int d(x,x_0)d\mu(x)$.
\item
If there exist $x_0\in E$, $a>0$ and $b\geq 1$ such that
\begin{equation}\label{leiden2}
\int e^{ a d(x_0,x)^2} d\mu(x) \leq b,
\end{equation}
then $\mu$ satisfies $\gcbl{D}$ with
\begin{equation}\label{dformule}
D=\frac{b^2e}{2a\sqrt{\pi}}.
\end{equation}
\end{enumerate}
\end{theorem}
\vskip0.3cm

Firstly, we establish the following crucial exponential estimate for the solution.
\begin{lemma}\label{prop exp est}
Assume that ({\bf H1}), ({\bf H2'}), ({\bf H3}), ({\bf H4'}), ({\bf H5}) and \eqref{local p} hold with $\alpha=2$ and a deterministic function $\tilde{f}$. Furthermore, we assume that \eqref{cons-con} holds.
Then for any $0<c<1-\frac{K_3}{\theta\eta}$ and $\lambda_0\in\left(0,\frac{(1-c)\theta \eta-K_3}{2C_B+\theta \eta}\right)$, we have that for any $x \in H$
\begin{equation}\label{eq exp1}
\mathbb E^x\left[\exp\left( c\lambda_0\theta \int_0^t \|X_s\|_V^{2}d s \right)  \right]\le e^{\lambda_0 \int_0^t \tilde{f}(s)ds}\cdot e^{\lambda_0\|x\|_H^2}.
\end{equation}

\end{lemma}
\begin{proof}
For any $0<c<1-\frac{K_3}{\theta\eta}$, let
$$
Y_t:=\|X_t\|_H^2+c\theta\int_0^t\|X_s\|_V^{2}d s,
$$
where $\theta$ is the constant appeared in {\bf (H3)}.

By It\^o's formula and ({\bf H3}), we have
\begin{align*}
d Y_t   =& \left(2 _{V^*}\langle A(t, X_t), X_t\rangle_{V}+\|B(t,X_t)\|_{\mathcal L_2(U;H)}^2 \right)dt+2\langle X_t, B(t,X_t)d W_t\rangle_H+c\theta\|X_t\|_V^{2}d t \notag \\
   \le& \left(\tilde{f}_t-(1-c)\theta\|X_t\|_V^{2}+K_3\|X_t\|_H^2 \right)d t+2\langle X_t, B(t,X_t)d W_t\rangle_H.
 \end{align*}
Denoting by $d[Y,Y]_t$ the quadratic variation process of a semimartingale $Y$, we can also compute with It\^o's formula to get
\begin{align}\label{Yt}
 d e^{\lambda_0Y_t}   =& e^{\lambda_0Y_t} \left(\lambda_0 dY_t+\frac{\lambda_0^2}{2}d [Y,Y]_t\right)\notag\\
 \le & \lambda_0 e^{\lambda_0Y_t} \left(\tilde{f}_t-(1-c)\theta\|X_t\|_V^{2}+K_3\|X_t\|_H^2+2\lambda_0\|B (t,X_t)\|_{\mathcal L_2(U;H)}^2\cdot\|X_t\|_H^2 \right)d t\notag \notag \\
 &+ 2\lambda_0e^{\lambda_0Y_t}\langle X_t, B(t,X_t)d W_t\rangle_H.
 \end{align}

Let $Z_t:=e^{-\lambda_0 \int_0^t \tilde{f}(s)ds}e^{\lambda_0 Y_t}$.
In view of  \eqref{Yt}, ({\bf H5}) and It\^o's formula again,
\begin{align*}
d Z_t=&e^{-\lambda_0 \int_0^t \tilde{f}(s)ds}d e^{\lambda_0 Y_t}+e^{\lambda_0 Y_t}d  e^{-\lambda_0 \int_0^t \tilde{f}(s)ds}\\
\le & \lambda_0 Z_t\left( -(1-c)\theta\|X_t\|_V^{2}+K_3\|X_t\|_H^2+2\lambda_0 C_B\|X_t\|_H^2 \right)d t+2\lambda_0Z_t\langle X_t, B(t,X_t)dW_t\rangle_H.
\end{align*}
Thus for any $0<c<1-\frac{K_3}{\theta\eta}$, when $0<\lambda_0<\frac{(1-c)\theta \eta-K_3}{2C_B}$, \eqref{norm-rel} implies
\begin{align*}
d Z_t &\le \lambda_0 Z_t\left( -(1-c)\theta\eta+K_3+2\lambda_0 C_B\right)\|X_t\|_H^2d t+2\lambda_0Z_t\langle X_t, B(t,X_t)dW_t\rangle_H\\
&\le 2\lambda_0Z_t\langle X_t, B(t,X_t)dW_t\rangle_H.
\end{align*}
Since $Z_t\ge0$, we obtain by Fatou's lemma $\mathbb E^x[Z_t]\le \mathbb E^x[Z_0]$, which is stronger than \eqref{eq exp1}.

The proof is complete.
\end{proof}

\begin{proof}[The proof of Theorem \ref{T1}]
In view of Theorem \ref{BG-T} and \ref{thmgaussianexpmoment}, to prove the probability measure $P^\mu$ satisfies ${\bf T_1}(C)$ for any $\mu\in \mathcal M_{\lambda_0}$ with $0<\lambda_0<\frac{(1-c)\theta \eta-K_3}{2C_B}$ and $0<c<1-\frac{K_3}{\theta\eta}$, it is enough to prove that $P^\mu$ has a Gaussian moment bound, i.e. Eq.\eqref{leiden2} holds for some constant $a>0$, $b\geq1$ and $x_0 \in H$. By Lemma \ref{prop exp est}, we see that for any $x\in H$

\begin{equation}\label{Px}
\int e^{ a d(x_0,x)^2} dP^x \leq b,
\end{equation}
where $a=c\lambda_0\theta$, $b=e^{\lambda_0\int_0^T\tilde{f}(s)ds+\lambda_0\|x\|_H^2}$ and $x_0=0 \in V$.
 Thus for any $\mu\in \mathcal M_{\lambda_0}$ with $0<\lambda_0<\frac{(1-c)\theta \eta-K_3}{2C_B}$ and $0<c<1-\frac{K_3}{\theta\eta}$,
\begin{align*}\label{Pu}
\int e^{ a d(x_0,x)^2} dP^\mu&=\int_H \int e^{ a d(x_0,x)^2} dP^x \mu(dx)\\
&\leq \int_H b \mu(dx)=e^{\lambda_0\int_0^T\tilde{f}(s)ds} \int_H e^{\lambda_0\|x\|_H^2}\mu(dx):=\tilde{b}.
\end{align*}
Taking into account that  $\tilde{f}, \lambda_0$ are positive and $\mu\in \mathcal M_{\lambda_0}$, we have $\tilde{b}\in(1,+\infty)$.
Thus Theorem \ref{thmgaussianexpmoment} implies that $P^\mu \in {\bf T_1}(C)$ with
$$
C=\frac{{\tilde{b}}^2e}{a\sqrt{\pi}}=\frac{\exp\left(2\lambda_0\int_0^T \tilde{f}(s)ds+1\right)}{c\lambda_0 \theta \sqrt{\pi}} \left(\int_H e^{\lambda_0\|x\|_H^2}\mu(dx)\right)^2.
$$

The proof is complete.
\end{proof}

\subsection{Examples}
This general framework includes a large number of new applications (see \cite{LR2010}), here we give two of them: the stochastic Burgers type equation and the stochastic $2$-D Navier-Stokes equation.

\bexa[Stochastic Burgers equation \cite{LR2010}]

Let $\Lambda=[0,1]$ and $W^{1,2}_0(\Lambda)$ be the standard Sobolev space. We consider the following triple
$$
V :=W^{1,2}_0(\Lambda)\subseteq H :=L^2(\Lambda)\subseteq (W^{1,2}_0(\Lambda))^\ast,
$$
with $\|v\|_V:=\left(\int_0^1\left(\frac{\partial v}{\partial x} \right)^2 dx\right)^{1/2}$ for any $v\in V$,
and the classical stochastic Burgers equation
\beq\label{SBE}
dX_t=\left(\Delta X_t+X_t\frac{\partial X_t}{\partial x}\right)dt+B(X_t)dW_t,
\nneq
where $W_t$ is a Wiener process on $H$. By the Poincar\'{e} inequality we know that $\eta$ in \eqref{norm-rel} is $\sqrt{\pi^2-1}$ in this case.

According to \cite{LR2010}, if $B:V \to \mathcal L_2(H)$  is Lipschitz,  conditions {\bf (H1)}, {\bf (H2')}, {\bf (H3)}, {\bf(H4')} hold for Eq.\eqref{SBE} with $\alpha=2$, $\beta=2$, $\theta=\frac{3}{2}$ and $\rho(v)=\|v\|^4_{L^4(\Lambda)}$.
Following the estimate on $[0,1]$ (see \cite[Lemma 2.1]{MS})
\beq\label{norest}
\|v\|_{L^4(\Lambda)}^4\leq 4 \|v\|_{L^2(\Lambda)}^2\left\|\frac{\partial v}{\partial x} \right\|_{L^2(\Lambda)}^2,
\nneq
we see that condition \eqref{local p} in Lemma \ref{local solu} holds for equation  \eqref{SBE}, thus with the initial value properly chosen, Eq.\eqref{SBE} has a unique solution.

If $B$ further satisfies $({\bf H5})$, then we have $K_3=0$ and $\tilde{f}(t)= \sup_{u\in H} \|B(u)\|_{\mathcal L_2(H)}^{2}=C_B$ in {\bf (H3)}. Hence by Theorem \ref{T1}, for any $c\in(0,1)$, when $\mu\in \mathcal M_{\lambda_0}$ with $\lambda_0\in\left(0,\frac{3(1-c)\sqrt{\pi^2-1}}{4C_B}\right)$, the law $P^\mu$ of the solution satisfies the ${\bf T_1}(C)$ property on $L^2([0,T];V)$ w.r.t. the $L^2$-metric with the constant given by \eqref{T1const}.
\nexa

\bexa[Stochastic $2$-D Navier-Stokes equation]
Let $\Lambda=(0,1)\times(0,1)$ be the rectangle  in $\rr^2$. Define
$$
V=\{v\in W^{1,2}_0(\Lambda,\rr^2): \nabla \cdot v=0\ a.s.\ in\ \Lambda\},\ \ \|v\|_V:=\left(\int_{\Lambda}|\nabla v|^2 dx \right)^{1/2},
$$
and $H$ is the closure of  $V$ in the  norm $\|v\|_H:=\left(\int_{\Lambda}|v|^2dx\right)^{1/2}$. By the Poincar\'{e } inequality, it is easy to see that  $\eta$ in \eqref{norm-rel} is $\sqrt{2\pi^2-1}$ in this case.

Define the linear operator $P_H$ (Helmhotz-Hodge projection) and $A$ (Stokes operator with viscosity constant $\nu$) by
$$
P_H:L^2(\Lambda,\rr^2)\to H\ \text{orthogonal projection};
$$
$$
A: W^{2,2}(\Lambda,\rr^2)\cap V \to H, Au=\nu P_H\triangle u.
$$
Consider the Gelfand triple
$$
V\subseteq H\equiv H^\ast\subseteq V^\ast,
$$
and the stochastic $2$-D Navier-Stokes equation
\begin{equation}\label{NS-2}
dX_t=(AX_t+F(X_t)+f_t)dt+B(X_t)dW_t,
\end{equation}
where $f\in L^2([0,T];V^\ast)$ denotes some external force,
$$
F:H\times V\to H,\ \ F(u,v)=-P_H[(u\cdot \nabla)v],\ \ F(u)=F(u,u)
$$
and $W_t$ is a Wiener process on $H$.


According to \cite{LR2010}, if $B$ satisfies
$$
\left\|B\left(v_1\right)-B\left(v_2\right)\right\|_{\mathcal L_2(H)}^{2} \leq K\left(1+\left\|v_2\right\|_{L^{4}(\Lambda)}^{4}\right)\left\|v_1-v_2\right\|_{H}^{2}, \quad v_1, v_2 \in V,
$$
where $K$ is some constant, then {\bf (H1)}, {\bf (H2')}, {\bf (H3)}, {\bf(H4')} hold for Eq.\eqref{NS-2} with $\alpha=2$, $\beta=2$, $\theta=\nu$ and $\rho(v)=\|v\|^4_{L^4(\Lambda)}$.
Similarly, following the well-known estimate on $\rr^2$ (\cite[Lemma 2.1]{MS}),
$$
\|v\|_{L^4(\Lambda)}^4\leq 2 \|v\|_{L^2(\Lambda)}^2\|\nabla v\|_{L^2(\Lambda)}^2,
$$
we see that condition \eqref{local p} in Lemma \ref{local solu} holds for Eq.\eqref{NS-2}, thus with the initial value properly chosen, equation \eqref{NS-2} has a unique solution. If $B$ further satisfies $({\bf H5})$, we have $K_3=0$ and $\tilde{f}(t)= \frac{\|f_t\|_{V^\star}^2}{\nu}+C_B$ in {\bf (H3)}.
Hence by Theorem \ref{T1}, for any $c\in(0,1)$, when $\mu\in \mathcal M_{\lambda_0}$ with $\lambda_0\in\left(0,\frac{(1-c)\nu\sqrt{2\pi^2-1}}{2C_B}\right)$, the law $P^\mu$ of the solution to Eq.\eqref{NS-2}  satisfies the ${\bf T_1}(C)$ property on $L^2([0,T];V)$ w.r.t. the $L^2$-metric with the constant given by \eqref{T1const}.

\nexa
%
%
%


{\bf Acknowledgement.} Ruinan Li is supported by Shanghai Sailing Program (Grant No. 21YF1415300) and NNSFC (12101392).


\begin{thebibliography}{abc}

 \bibitem{BG} Bobkov, S.G., G\"otze F., 1999. Exponential integrability and transportation cost related to logarithmic Sobolev inequalities.
 J. Funct. Anal., 163(1), 1-28.

\bibitem{BGL2001} Bobkov, S.G., Gentil, I., Ledoux, M., 2001. Hypercontractivity of Hamilton-Jacobi equations.
J Math Pures Appl, 80(7), 669-696.

\bibitem{BH2000} Bobkov, S.G., Houdre, C., 2000. Weak dimension-free concentration of measure. Bernoulli, 6(4), 621-632.

\bibitem{BLM}  Boucheron, S., Lugosi, G., Massart,P., 2013. Concentration
Inequalities: A Nonasymptotic Theory of Independence, Oxford University Press.

\bibitem{BH} Boufoussi, B., Hajji, S., 2018. Transportation inequalities for stochastic heat equations.
Statist. Probab. Lett., 139, 75-83.

\bibitem{CCR} Chazottes, J.R., Collet, P., Redig, F., 2020.
  Evolution of Gaussian concentration bounds under diffusions, arXiv:1903.07915v2.

\bibitem{DRR} Da Prato, G., Rozovskii, B., R\"{o}ckner, M.,  Wang, F.-Y., 2006. Strong solutions of stochastic generalized porous media equations: Existence, uniqueness and ergodicity. Comm Partial Differential Equations, 31, 277-291.

\bibitem{DL}Dai, Y., Li, R., 2022. Transportation Inequalities for Stochastic Heat Equation with Rough Dependence in Space. Acta. Math. Sin. English Ser. 38, 2019šC2038.


\bibitem{DGW}Djellout, H., Guillin, A., Wu, L., 2004. Transportation cost-information inequalities and applications to random dynamical systems and diffusions. Ann. Probab., 32(3B), 2702-2732.

\bibitem{DP} Dubhashi, D.P., Panconesi, A. 2012. Concentration of Measure for the Analysis of Randomized Algorithms. Cambridge University Press, Cambridge.



\bibitem{Goz} Gozlan, N., 2015. Transport inequalities and concentration of measure. ESAIM Proc. Surveys, 51, 1-23.

\bibitem{KA} Kavin, R., Ananta, K. Majee, 2022. Stochastic evolutionary $p$-Laplace equation: large deviation principles and transportation cost inequality, arXiv:2210.11036.


\bibitem{KS} Khoshnevisan, D., Sarantsev, A., 2019. Talagrand concentration inequalities for stochastic partial differential equations. Stoch PDE: Anal. Comp., 7(4), 679-698.


\bibitem{KR}Krylov,N., Rozovskii, B., 2007. Stochastic evolution equations. In: Stochastic Differential Equations: Theory and Applications.
Interdisciplinary Mathematical Sciences, vol. 2. Hackensack: World Scientific, 1-69.
%


\bibitem{La} Lacker, D., 2018. Liquidity, risk measures, and concentration of measure. Math. Oper. Res., 43(3), 813-837.

\bibitem{Ledoux01} Ledoux, M., 2001. The Concentration of Measure Phenomenon. Mathematical Surveys and Monographs, 89. American Mathematical Society, Providence, RI.

\bibitem{Ledoux03} Ledoux, M., 2003. Measure Concentration, Transportation Cost, And Functional Inequalities. summer school on singular phenomena $\&$ scaling in mathematical models.


\bibitem{Liu} Liu, W., 2009. On the stochastic $p$-Laplace equation. J. Math. Anal. Appl., 360, 737-751.

\bibitem{LW1} Li, Y., Wang, X., 2020. Transportation cost-information inequality for stochastic wave equation. Acta Appl. Math., 169, 145-155.

\bibitem{LR2010} Liu, W., R\"{o}ckner, M., 2010. SPDE in Hilbert space with locally monotone coefficients. J. Funct. Anal.,  259, 2902-2922.

\bibitem{LR}Liu, W.,  R\"ockner, M., 2015. Stochastic Partial Differential Equations: An Introduction.  Universitext, Springer, New York.

\bibitem{LW} Liu, W, Wang F.-Y., 2008. Harnack inequality and strong Feller property for stochastic fast-diffusion equations. J. Math. Anal.
Appl., 342, 651-662.

\bibitem{MW} Ma, Y., Wang, R., 2020.
Transportation cost inequalities for stochastic reaction-diffusion equations with L\'evy noises and non-Lipschitz reaction terms,   Acta. Math.  Sin. Engl. Ser., 36(2), 121-136.


\bibitem{Marton96} Marton, K., 1996.  Bounding $d$-distance by information divergence: A method to prove
measure concentration. Ann. Probab. 24(2), 857-866.

\bibitem{Marton97} Marton, K., 1996. A measure concentration inequality for contracting Markov chains.
Geom. Funct. Anal. 6(3), 556-571.

\bibitem{M} Massart, P., 2007. Concentration Inequalities and Model Selection. Lecture Notes in Mathematics, 1896. Springer, Berlin.

\bibitem{MS}
Menaldi, J.-L., Sritharan, S.S., 2002. Stochastic $2$-D Navier-Stokes equation, Appl. Math. Optim., 46, 31-53.

\bibitem{OV}  Otto, F., Villani, C., 2000. Generalization of an inequality by Talagrand and links with
the logarithmic Sobolev inequality. J. Funct. Anal., 173(2), 361-400.


\bibitem{PR} Pr\'ev\^ot, C.,  R\"ockner,  M., 2007. A Concise Course on Stochastic Partial Differential Equations,  Lecture Notes in Mathematics, 1905, Springer, Berlin.

\bibitem{RRW} Ren, J, R\"{o}ckner M, Wang F.-Y., 2007.  Stochastic generalized porous media and fast diffusion equations. J. Differential Equations, 238, 118-152.

\bibitem{SW}
Shang, S.,  Wang, R., 2020. Transportation inequalities under uniform metric for a stochastic heat equation driven by time-white and space-colored noise. Acta. Appl. Math., 170, 81-97.

\bibitem{SZ} Shang, S., Zhang, T., 2019. Talagrand concentration inequalities for stochastic heat-type equations under uniform distance. Electron. J. Probab., 24(129), 1-15.

\bibitem{Tal0} Talagrand, M., 1995. Concentration of measure and isoperimetric inequalities
in product spaces. Publications Math\'{e}matiques de IH\'{E}S, 81(1), 73-205.

\bibitem{Tal1} Talagrand, M., 1996. Transportation cost for Gaussian and other product measures. Geom. Funct. Anal., 6(3), 587-600.

\bibitem{Tal2} Talagrand, M., 1996.  New concentration inequalities in product spaces. Inventiones Mathematicae, 126(3), 505-563.


\bibitem{Villani} Villani, C., 2009.  Optimal Transport: Old and New. Grundlehren der Mathematischen Wissenschaften, 338. Springer, Berlin.


\bibitem{WZ} Wang, F.-Y.,  Zhang, T., 2020. Talagrand inequality on free path space and application to stochastic reaction diffusion equations. Acta Mathematicae Applicatae Sinica, English Series, 36(2), 253-261.


\bibitem{WZ2006}
Wu, L., Zhang, Z., 2006. Talagrand's $T_2$-transportation inequality and log-Sobolev inequality for dissipative SPDEs and applications to reaction-diffusion equations. Chinese Ann. Math. Ser. B, 27(3), 243-262.


\end{thebibliography}
\end{document}